\documentclass{article}

\usepackage{amsmath,amssymb,amsthm}
\providecommand{\R}{\mathbb R}
\providecommand{\drm}{\mathrm d}
\providecommand{\E}{\mathsf E}
\providecommand{\prob}{\mathsf P}

\providecommand{\tr}{\mathrm{tr}}
\providecommand{\e}{\mathrm e}
\providecommand{\sca}[1]{\langle #1 \rangle}
\providecommand{\abs}[1]{\lvert #1 \rvert}
\providecommand{\norm}[1]{\lVert #1 \rVert}

\providecommand{\det}{\mathop{\mathrm{det}}\nolimits}
\providecommand{\tr}{\mathop{\mathrm{tr}}\nolimits}

\newtheorem{thm}{Theorem}
\newtheorem{lem}[thm]{Lemma}
\newtheorem{prop}[thm]{Proposition}

\theoremstyle{remark}
\newtheorem*{rem}{Remark}
\newtheorem*{ex}{Example}

\theoremstyle{definition}
\newtheorem{defi}[thm]{Definition}

\title{Short probabilistic proof of the Brascamp-Lieb  and Barthe theorems}
\author{Joseph Lehec\footnote{Universit\'e Paris-Dauphine}}

\begin{document}

\maketitle

\section{Introduction}
A Brascamp-Lieb datum on $\R^n$ is a finite sequence
\begin{equation}
\label{data}
(c_1,B_1 ) , \dotsc , (c_m,B_m) 
\end{equation}
where $c_i$ is a positive number and $B_i \colon \R^n \to \R^{n_i}$ is linear and onto.
The Brascamp-Lieb constant associated to this datum is the smallest real number $C$ 
such that the inequality
\begin{equation}
\label{BL}
\int_{\R^n} \prod_{i=1}^m ( f_i \circ B_i )^{c_i}  \, \drm x
\leq C \prod_{i=1}^m \bigl( \int_{\R^{n_i}} f_i \, \drm x \bigr)^{c_i} 
\end{equation}
holds for every set of non-negative integrable functions $f_i\colon \R^{n_i} \to \R$.
The Brascamp-Lieb theorem~\cite{bl,lieb} asserts that~\eqref{BL}
is saturated by Gaussian functions. In other words
if~\eqref{BL} holds for every functions $f_1,\dotsc,f_m$ of the form
\[
f_i ( x ) = \e^{- \sca{ A_i x ,  x } / 2} 
\]
where $A_i$ is a symmetric positive definite matrix on $\R^{n_i}$
then~\eqref{BL} holds for every set of functions $f_1,\dotsc,f_m$.

The reversed Brascamp-Lieb constant associated to~\eqref{data}
is the smallest constant $C_r$ 
such that for every non-negative measurable functions $f_1,\dotsc,f_m,f$
satisfying
\begin{equation}
\label{RBL-hyp}
\prod_{i=1}^m f_i ( x_i )^{c_i} \leq f \Bigl( \sum_{i=1}^m c_i B_i^* x_i \Bigr)
\end{equation}
for every $(x_1 ,\dotsc, x_m ) \in \R^{n_1}\times \dotsb\times\R^{n_m}$ we have
\begin{equation}
\label{RBL}
 \prod_{i=1}^m \bigl( \int_{\R^{n_i}} f_i \, \drm x\bigr)^{c_i}
 \leq C_r \int_{\R^n} f   \, \drm x .
\end{equation}
It was shown by Barthe~\cite{barthe} that again Gaussian functions
saturate the inequality. The original paper of Brascamp and Lieb~\cite{bl}
rely on symmetrization techniques. Barthe's argument uses optimal transport
and works for both the direct and the reversed inequality. 
More recent proofs of the direct inequality~\cite{bbc,bcct,cc,cll}
all rely on semi-group techniques. 
Barthe and Huet~\cite{bh} have a semi-group 
argument that works for both the direct and reversed inequality, 
provided the Brascamp-Lieb datum satisfies 
\begin{equation}
\label{frame}
\begin{split}
& B_i B_i^* = \mathrm{id}_{\R^{n_i}} , \ \forall i \leq m,\\
& \sum_{i=1}^m c_i B_i^* B_i = \mathrm{id}_{\R^n} .
\end{split}
\end{equation}
This constraint is called the \emph{frame condition} hereafter. 
\\
The purpose of this article is to give a short
probabilistic proof of the Brascamp-Lieb and Barthe theorems.
 Our main tool shall be a representation formula for the quantity
\[ \ln \Bigl( \int \e^{g ( x ) } \ \gamma ( \drm x)  \Bigr) , \]
where $\gamma$ is a Gaussian measure. Let us describe it briefly. 
Let $(\Omega,\mathcal A,\prob)$
be a probability space, let $(\mathcal F_t )_{t\in [0,T]}$
be a filtration and let 
\[
(W_t)_{t \in [0,T]}
\]
be a Brownian motion taking values in $\R^n$ (we fix a finite time horizon $T$).
Assuming that the covariance matrix $A$ of $W$ (i.e. the covariance
matrix of the random vector $W_1$) has full rank,
we let $\mathbb H$ be the associated Cameron-Martin space;
namely the Hilbert space of absolutely continuous paths $u\colon [0,T] \to \R^n$
starting from $0$, equipped with the norm
\[
\norm{u}_{\mathbb H} = \bigl(  \int_0^{T} \sca{ A^{-1} \dot u_s ,  \dot u_s } \, \drm s  \bigr)^{1/2}.
\]
In the sequel we call \emph{drift} any adapted
process $U$ which belongs to $\mathbb H$ almost surely.
The following formula is due to Bou\'e and Dupuis~\cite{boue-dupuis}
(see also \cite{borell,lehec}).
\begin{prop}
\label{borell}
Let $g \colon \R^n \to \R$ be measurable and bounded 
from below, then
\[
\log \Bigl( \E \e^{ g ( W_T ) } \Bigr) 
= \sup \Bigl[ \E \Bigl( g ( W_T + U_T ) - \frac{1}{2} \norm{U}^2_{\mathbb H} \Bigr) \Bigr] 
\]
where the supremum is taken over all drifts $U$.
\end{prop}
In~\cite{borell}, Borell rediscovers this formula
and shows that it yields the Pr\'ekopa-Leindler inequality
(a reversed form of H\"older's inequality) very easily. 
Later on Cordero and Maurey noticed that under the frame condition,
both the direct and reversed Brascamp-Lieb inequalities could be recovered this way
(this was not published but is explained in~\cite{lehec}). 
The purpose of this article is, following Borell, Cordero and Maurey,
to show that the Brascamp-Lieb and Barthe theorems in full generality 
are direct consequences of Proposition~\ref{borell}. 
\section{The direct inequality}
%
%
%
Replace $f_i$ by $x\mapsto f_i (  x / \lambda)$
in inequality~\eqref{BL}. The left-hand side of the inequality
is multiplied by $\lambda^n$ and the right-hand side by $\lambda^{\sum_{i=1}^m c_i n_i}$.
Therefore, a necessary condition for $C$ to be finite is
\[
\sum_{i=1}^m c_i n_i = n .
\]
This homogeneity condition will be assumed throughout the rest of the article. 
\begin{thm}
\label{bl-thm}
Assume that there exists a matrix $A$
satisfying
\begin{equation}
\label{john}
A^{-1} = \sum_{i=1}^m c_i B_i^* ( B_i A B_i^* )^{-1} B_i . 
\end{equation}
Then the Brascamp-Lieb constant is
\[
C = \Bigl( \frac { \det ( A )   }{ \prod_{i=1}^m \det ( B_i A B_i^* )^{c_i} }   \Bigr)^{1/2} ,
\]
and there is equality in~\eqref{BL} for the following Gaussian functions
\begin{equation}
\label{blequality}
f_i \colon  x\in \R^{n_i} \mapsto \e^{ - \sca{ (B_i AB_i^*)^{-1} x , x} /2  }  , \quad  i\leq m.
\end{equation}
\end{thm}
\begin{rem}
If the frame condition~\eqref{frame} holds then 
$A = \mathrm{id}_{\R^n}$ satisfies~\eqref{john}
and the Brascamp-Lieb constant is $1$. 
\end{rem}
\begin{proof}
Because of~\eqref{john}, if the functions $f_i$ are defined by~\eqref{blequality}
then
\[
\prod_{i=1}^m \bigl( f_i (B_i x) \bigr)^{c_i}  = \e^{-\sca{A^{-1} x, x} / 2 } .
\]
The equality case follows easily (recall the homogeneity condition $\sum c_i n_i =n$).
\\
Let us prove the inequality. Let 
$f_1,\dotsc,f_m$ be non-negative integrable functions on $\R^{n_1},\dotsc,\R^{n_m}$,
respectively and let 
\[
f\colon x\in \R^n \mapsto \prod_{i=1}^m f_i (B_i x)^{c_i} . 
\]
Fix $\delta >0$, let $g_i = \log ( f_i +\delta)$ for every $i\leq m$ and let
\[
g (x) = \sum_{i=1}^m  c_i g_i ( B_i x ) .
\] 
The functions $(g_i)_{i\leq m}, g$ are bounded from below. 
Fix a time horizon $T$, let $(W_t)_{t\geq T}$ be a Brownian
motion on $\R^n$, starting from $0$ and having covariance $A$; and
let  $\mathbb H$ be the associated Cameron-Martin space. 
By Proposition~\ref{borell}, given $\epsilon>0$, there exists a drift $U$ such that
\begin{equation}
\label{BLetape1}
\begin{split}
\log \Bigl( \E \e^{ g ( W_T ) } \Bigr) 
& \leq \E \Bigl(  g ( W_T + U_T ) - \frac{1}{2} \norm{ U }^2_{\mathbb H} \Bigr) + \epsilon \\
& = \sum_{i=1}^m c_i \E   g_i ( B_i W_T + B_i U_T )  
  - \frac{1}{2} \E  \norm{U}_{\mathbb H}^2 + \epsilon   .
\end{split}
\end{equation}
The process $B_i W$ is a Brownian motion on $\R^{n_i}$ with covariance 
$B_i A B_i^*$. Set $A_i = B_i A B_i^*$ and let $\mathbb H_i$ be the Cameron-Martin space associated to $B_i W$. 
Equality~\eqref{john} gives
\[
\sca{ A^{-1} x , x  } = \sum_{i=1}^m c_i \sca{ A_i^{-1}  B_i x , B_i x } 
\]
for every $x\in \R^n$. This implies that 
\[
\norm{ u }_{\mathbb H}^2 =  \sum_{i=1}^{m} c_i  \norm{B_ i u}_{\mathbb H_i}^2 
\]
for every absolutely continuous path $u\colon [0,T] \to \R^n$.
So that~\eqref{BLetape1} becomes
\[
\log \Bigl( \E \e^{ g ( W_T ) } \Bigr) 
\leq \sum_{i=1}^m c_i \E \Bigl(   g_i ( B_i W_T + B_i U_T )  
  - \frac{1}{2} \norm{ B_i U }_{\mathbb H_i}^2 \Bigr)
  + \epsilon .
 \]
By Proposition~\ref{borell} again we have 
\[
\E \Bigl(   g_i ( B_i W_T + B_i U_T )  - \frac{1}{2} \norm{ B_i U }_{\mathbb H_i}^2 \Bigr)  
\leq 
\log \Bigl( \E \e^{ g_i( B_i W_T ) } \Bigr)
\]
for every $i\leq m$. We obtain (dropping $\epsilon$ which is arbitrary) 
\begin{equation}
\label{BLetape2}
\log \Bigl( \E \e^{ g ( W_T ) } \Bigr) 
\leq \sum_{i=1}^m c_i \log \Bigl( \E \e^{ g_i( B_i W_T ) } \Bigr) .
\end{equation}
Recall that $f \leq \e^g$ and observe that 
\[
\prod_{i=1}^m \Bigl( \E ( \e^{g_i} ( B_i W_T )  \Bigr)^{c_i} 
\leq \prod_{i=1}^m \Bigl( \E f_i ( B_i W_T )  \Bigr)^{c_i}  + O(  \delta^c ) ,
\]
for some positive constant $c$. 
Inequality~\eqref{BLetape2} becomes (dropping the $O(\delta^c)$ term)
\begin{equation}
\label{BLetape3}
\E  f ( W_T )  
\leq \prod_{i=1}^m \Bigl( \E f_i ( B_i W_T )  \Bigr)^{c_i} .
\end{equation}
Since $W_T$ is a centered Gaussian vector with covariance $T A$
\[
\E f ( W_T )   =  \frac 1 {  (2\pi T)^{n/2} \det (A)^{1/2} }  
\int_{\R^n} f  (x) \e^{ - \sca{ A^{-1} x , x } / 2 T  } \, \drm x ,
\]
and there a similar equality for $\E  f_i( B_i W_T )$.
Then it is easy to see that letting $T$ tend to $+\infty$ in inequality~\eqref{BLetape3} 
yields the result (recall that $\sum c_i n_i =n$). 
 \end{proof}
\begin{ex}[Optimal constant in Young's inequality]
Young's convolution inequality asserts that if $p,q,r \geq 1$
and are linked by the equation 
\begin{equation}
\label{contrainteyoung}
\frac 1 p + \frac 1 q = 1+ \frac 1 r  ,
\end{equation}
then
\[
\norm{ F * G }_r \leq  \norm{ F }_p \norm{ G }_q ,
\]
for all $F \in L_p$ and $G \in L_q$. 
When either $p$, $q$ or $r$ equals $1$ or $+\infty$
the inequality is a consequence of H\"older's inequality
and is easily seen to be sharp.
On the other hand when $p,q,r$ belong to
the open interval $(1,+\infty)$
the best constant $C$ in the inequality
\[
\norm{ F * G }_r \leq  C \norm{ F }_p \norm{ G }_q ,
\]
is actually smaller than $1$. Let us compute it using the previous theorem. 
Observe that by duality $C$ is the best constant in the inequality
\begin{equation}
\label{blyoung}
\int_{\R^2 } f^{c_1} (x+y) g^{c_2} (y) h^{c_3} (x) \, \drm x \drm y \leq 
C
\Bigl( \int_\R f \Bigr)^{c_1}
\Bigl( \int_\R g \Bigr)^{c_2}
\Bigl( \int_\R h \Bigr)^{c_3} ,
\end{equation}
where 
\[
c_1 = \frac 1 p , \ c_2 = \frac 1 q , \ c_3 = 1 - \frac 1 r .
\] 
In other words $C$ is the Brascamp-Lieb constant 
in $\R^2$ associated to the data
\[
( c_1 , B_1 ) , ( c_2  ,  B_2  ) , ( c_3 , B_3 ) ,
\]
where $B_1 = (1,1)$, $B_2 = (0,1)$ and $B_3 = (1,0)$. 
According to the previous result, we have to find a positive definite matrix $A$
satisfying
\[
A^{-1} = \sum_{i=1}^3  c_i B_i^* ( B_i A B_i^* )^{-1} B_i . 
\] 
Letting 
$
A = 
\begin{pmatrix} 
x & z \\
z & y
\end{pmatrix}
$, this equation turns out to be equivalent to
%
\[
\begin{split}
(1- c_2 ) xy + yz  + c_2 z^2 & = 0 \\
(1- c_3 ) xy + xz + c_3 z^2 & = 0 \\
c_1 + c_2 + c_3 & = 2  .
\end{split}
\]
The third equation is just the Young constraint~\eqref{contrainteyoung}.
The first two equations admit two families of solutions: either $(x,y,z)$
is a multiple of $(1,1,-1)$ or $(x,y,z)$ is a multiple of
 \[
  \bigl( c_3 ( 1- c_3 ) , c_2  ( 1 - c_2 ) , - (1- c_2 ) (1 - c_3)  \bigr) .
 \]
The constraint $xy - z^2 >0$ rules out the first solution.
The second solution is fine since $c_1,c_2$ and $c_3$ are assumed to
belong to the open interval $(0,1)$. 
By Theorem~\ref{bl-thm}, the best constant in~\eqref{blyoung} is
\[
C = \Bigl( \frac{ \det ( A) }{ \prod_{i=1}^3 \det ( B_i A B_i^* )^{c_i} } \Bigr)^{ 1/2 } 
=
\Bigl(
  \frac{   (1- c_1 )^{1- c_1} (1- c_2 )^{1- c_2 } (1- c_3 )^{1-c_3} }
  { c_1^{c_1}  c_2^{c_2}c_3^{c_3}} \Bigr)^{1/2}  .
\]
In terms of $p,q,r$ we have
\[
C = \Bigl( \frac{  p^{1/p} \  q^{1/q}  \ {r'}^{1/  r' } }{  {p'}^ {1/p'}  \ {q'}^{1/q'} \  r^{1/r}  }  \Bigr)^{ 1/2 }
\]
where $p',q',r'$ are the conjugate exponents of $p,q,r$, respectively.
This is indeed the best constant in Young's inequality, 
first obtained by Beckner~\cite{beckner}.
\end{ex}

\section{The reversed inequality}
\begin{thm}
\label{rbl-thm}
Again, assume that there is a matrix $A$ satisfying~\eqref{john}. 
Then the reversed Brascamp-Lieb constant is
\[
 C_r =  \Bigl( \frac { \det ( A )   }{ \prod_{i=1}^m \det ( B_i A B_i^* )^{c_i} } \Bigr)^{1/2}  .
 \]
There is equality in~\eqref{RBL} for the following Gaussian functions
\[
\begin{split}
f_i & \colon  x\in \R^{n_i}  \mapsto \e^{ - \sca{ B_i A B_i^*  x , x} /2  }  , \quad  i\leq m. \\
f & \colon  x\in \R^{n}  \mapsto \e^{ - \sca{ A x , x} /2  }  .  \\
\end{split}
\]
\end{thm}
\begin{rem}
Observe that under condition~\eqref{john} the Brascamp-Lieb constant
and the reversed constant are the same, but the extremizers differ. 
\end{rem}
We shall use the following elementary lemma.
\begin{lem}
\label{hahn-banach}
Let $A_1 , \dotsc , A_m$ be positive definite matrices on 
$\R^{n_1}, \dotsc , \R^{n_m}$, 
respectively and let 
\[
A = \Bigl( \sum_{i=1}^m c_i B_i^* {A_i^{-1}} B_i \Bigr)^{-1} .
\]
Then for all $x\in \R^n$
\[
 \sca{ A x , x }  =
\inf  \Big\{ \sum_{i=1}^m c_i \sca{ A_i x_i  , x_i }  , \ \sum_{i=1}^m  c_i B_i^* x_i = x  \Bigr\} .
\]
\end{lem}
\begin{proof}
Let $x_1,\dotsc,x_m$ and let
\begin{equation}
\label{truc}
x = \sum_{i=1}^m c_i B_i^* x_i . 
\end{equation}
Then by the Cauchy-Schwarz inequality 
(recall that the matrices $A_i$ are positive definite)
\[
\begin{split}
\sca{ A x , x } 
& = \sum_{i=1}^m c_i \sca{ A x , B_i^* x_i } = \sum_{i=1}^m c_i \sca{ B_i A x , x_i } \\
& \leq \Bigl( \sum_{i=1}^m c_i  \sca{ A_i^{-1} B_i A x  , B_i A x } \Bigr)^{1/2}
 \Bigl( \sum_{i=1}^m c_i \sca{A_i x_i  , x_i} \Bigr)^{1/2} \\
& =  \sca{Ax , x }^{1/2} \ \Bigl( \sum_{i=1}^m c_i  \sca{ A_i x_i , x_i } \Bigr)^{1/2} .
\end{split}
\]
Besides, given $x\in \R^n$, set $x_i = A_i^{-1} B_i A x$ for all $i\leq m$. 
Then~\eqref{truc} holds and there is equality in the above Cauchy-Schwarz inequality. 
This concludes the proof. 
 \end{proof}
\begin{proof}[Proof of Theorem~\ref{rbl-thm}]
The equality case is a straightforward 
consequence of the hypotethis~\eqref{john} 
and Lemma~\ref{hahn-banach}, details are left to the
reader. 
\\
Let us prove the inequality. There is no loss of generality 
assuming that the functions $f_1,\dotsc,f_m$ are bounded from above
(otherwise replace $f_i$ by $\max( f_i , k)$, let $k$ tend to $+\infty$
and use monotone convergence). 
Fix $\delta >0$ and let $g_i = \log ( f_i +\delta)$ for every $i\leq m$.
By~\eqref{RBL-hyp} and since the functions $f_i$ are bounded from above,
 there exist positive constants $c,C$ such that the function
\[
g \colon x\in\R^n \mapsto \log \bigl(  f(x)  + C  \delta^c \bigr)  ,
\] 
satisfies 
\begin{equation}
\label{hypo-g}
 \sum_{i=1}^m c_i g_i ( x_i ) \leq g \Bigl( \sum_{i=1}^m c_i B_i^* x_i  \Bigr) 
\end{equation}
for every $x_1,\dotsc, x_m$. 
Observe that the functions $(g_i)_{i\leq m} , g$ are bounded from below. 
Let $(W_t)_{t\leq T}$ be a Brownian motion on $\R^n$ having covariance matrix $A$.
Set $A_i =  B_i A B_i^* $, then $A_i^{-1} B_i W$ is a Brownian motion
on $\R^{n_i}$ with covariance matrix
\[
( A_i^{-1} B_i ) A ( A_i^{-1} B_i )^* = A_i^{-1} ( B_i A B_i^* ) A_i^{-1} = A_i^{-1}.
\]
Let $\mathbb H_i$ be the associated Cameron-Martin space.
By Proposition~\ref{borell} there exists a ($\R^{n_i}$-valued) drift $U_i$ such that
\begin{equation}
\label{RBLetape1}
\log \Bigl( \E \e^{ g_i ( A_i^{-1} B_i W_T ) } \Bigr) 
\leq 
\E  
\Bigl(  g_i ( A_i^{-1} B_i W_T +  (U_i)_T ) - \frac 1 2 \norm{ U_i }_{\mathbb H_i}^2 \Bigr)  + \epsilon .
\end{equation}
By~\eqref{hypo-g} and~\eqref{john}
\[
\begin{split}
\sum_{i=1}^m c_i g_i ( A_i^{-1} B_i W_T + (U_i)_T ) &  
\leq g \Bigl( \sum_{i=1}^m c_i B_i^* ( A_i^{-1} B_i  W_T + (U_i)_T ) \Bigr) \\
 & =  g \Bigl( A^{-1} W_T + \sum_{i=1}^m c_i B_i^* (U_i)_T \Bigr) .
\end{split}
\]
The Brownian motion $(A^{-1} W)_{t\leq T}$ has covariance matrix $ A^{-1} A (A^{-1})^* = A^{-1}$. 
Let $\mathbb H$ be the associated Cameron-Martin space.
Lemma~\ref{hahn-banach} shows that
\[
\Bigl\langle A \Bigl( \sum_{i=1}^m c_i B_i^* x_i \Bigr) , \sum_{i=1}^m c_i B_i^* x_i \Bigr\rangle
\leq
\sum_{i=1}^m c_i  \sca{ A_i  x_i  ,  x_i }
\]
for every $x_1,\dotsc,x_m$ in $\R^{n_1},\dotsc,\R^{n_m}$, respectively. 
Therefore
\[
\Bigl\lVert \sum_{i=1}^m c_i B_i^* u_i \Bigr\rVert_{\mathbb H}^2 
\leq \sum_{i=1}^m c_i \norm{u_i}_{\mathbb H_i}^2 .
\]
for every sequence of absolutely continuous paths $( u_i \colon [0,T] \to \R^{n_i} )_{i\leq m}$.
Thus multiplying~\eqref{RBLetape1} by $c_i$ and summing over $i$ yields
\[
\begin{split}
\sum_{i=1}^m c_i &  \log \Bigl( \E \e^{ g_i ( A_i^{-1} B_i W_T ) } \Bigr) \\
& \leq \E \Bigl[ g \bigl(  A^{-1} W_T + \sum_{i=1}^m c_i B_i^* (U_i)_T \bigr)
  - \frac 1 2 \bigl\Vert \sum_{i=1}^m c_i B_i^* U_i \bigr\Vert_{\mathbb H}^2  \Bigr] +  \sum_{i=1}^m c_i  \epsilon .
\end{split}
\]
Hence, using Proposition~\ref{borell} again and dropping $\epsilon$ again,
\begin{equation}
\label{RBLetape2}
\sum_{i=1}^m c_i \log \Bigl( \E \e^{g_i  ( A_i^{-1} B_i W_T ) } \Bigr)^{c_i}
\leq \log \Bigl( \E \e^{ g ( A^{-1} W_T ) } \Bigr) .
\end{equation}
Recall that $f_i \leq \e^{g_i}$ for every $i\leq m$
and that $\e^{g}  = f + C  \delta^c$. Since $\delta$ is arbitrary, 
inequality~\eqref{RBLetape2} becomes
\[
\prod_{i=1}^m \Bigl( \E f_i  ( A^{-1}_i B_i W_T )  \Bigr)^{c_i} 
\leq 
\E f ( A^{-1} W_T )  .
\]
Again, letting $T$ tend to $+\infty$
in this inequality yields the result. 
 \end{proof}
%
%
\section{The Brascamp-Lieb and Barthe theorems}
%
%
%
So far we have seen that both the direct inequality and the reversed version
are saturated by Gaussian functions when there exists a matrix $A$ such that 
\begin{equation}
\label{john1}
A^{-1} = \sum_{i=1}^m c_i B_i^* ( B_i A B_i^* )^{-1} B_i . 
\end{equation}
In this section, we briefly explain why 
this yields the Brascamp-Lieb
and Barthe theorems. 
\\
Applying~\eqref{BL} to Gaussian functions gives
\begin{equation}
\label{BLG}
\prod_{i=1}^m \det ( A_i )^{c_i}   
\leq C^2 \det \bigl( \sum_{i=1}^m c_i B_i^* A_i B_i \bigr) ,
\end{equation}
for every sequence $A_1,\dotsc, A_m$
of positive definite matrices on $\R^{n_1},\dotsc,\R^{n_m}$.
Let $C_g$ be the Gaussian Brascamp-Lieb constant;
namely the best constant in the previous inequality. We 
have $C_g \leq C$ and it turns out that applying~\eqref{RBL}
to Gaussian functions
yields $C_g \leq C_r$ 
(one has to apply~Lemma~\ref{hahn-banach} at some point). 
\\
It is known since the work of 
Carlen and Cordero~\cite{cc}
that there is a dual formulation of~\eqref{BL}
in terms of relative entropy. 
In the same way, there is a dual formulation of~\eqref{BLG}.
For every positive matrix $A$ on $\R^n$, one has
\[
\log \det (A) = \inf_{B>0} \bigl( \tr ( AB ) - n - \log ( \det (B ) ) \bigr) ,
\]
with equality when $B = A^{-1}$. 
Using this and the equality $\sum_{i=1}^m c_i n_i = n$,
it is easily seen that 
$C_g$ is also the best constant such that  the inequality
\begin{equation}
\label{BLGD}
\det ( A ) \leq C_g^2 \prod_{i=1}^m \det ( B_i A B_i^* )^{c_i} 
\end{equation}
holds for every positive definite matrix $A$ on $\R^n$. 
\begin{ex}
Assume that $m=n$, that $c_1 = \dotsb = c_n =1$ and that $B_i (x) = x_i$
for $i\in [n]$. Inequality~\eqref{BLG} trivially holds with constant $1$
(and there is equality for every $A_1,\dotsc,A_n$).
On the other hand~\eqref{BLGD} becomes
\[
 \det ( A ) \leq \prod_{i=1}^n a_{ii} ,
\]
for every positive definite $A$, with equality when $A$ is diagonal. 
This is Hadamard's inequality. 
%
\end{ex}
\begin{lem}
\label{variational}
If $A$ is extremal in~\eqref{BLGD} then $A$ satisfies~\eqref{john1}. 
\end{lem}
\begin{proof}
Just compute the gradient of the map
\[
 A > 0  \mapsto \log \det ( A ) - \sum_{i=1}^m c_i \log \det ( B_i A B_i^* ) . \qedhere
\]
\end{proof}
Therefore, if the constant $C_g$ is finite and if 
there is an extremizer $A$ in~\eqref{BLGD} then $A$
satisfies~\eqref{john1} and together with the results 
of the previous sections we get the Brascamp-Lieb and Barthe
equalities
\begin{equation}
\label{bbl-thm}
C = C_r = C_g . 
\end{equation}
Although it may happen that $C_g <+\infty$ and no Gaussian extremizer exists,
there is a way to bypass this issue. 
For the Brascamp-Lieb theorem,
there is an abstract argument showing that is it is enough
to prove the equality $C = C_g$ when there is
a Gaussian extremizer. This argument relies on:
\begin{enumerate}
\item A criterion for having a Gaussian extremizer, due to Barthe~\cite{barthe}
in the rank $1$ case (namely when the dimensions $n_i$
are all equal to $1$) and Bennett, Carbery, Christ and Tao~\cite{bcct}
in the general case. 
\item A multiplicativity property of $C$ and $C_g$ 
due to Carlen, Lieb and Loss~\cite{cll} in the 
rank $1$ case and BCCT again in general. 
\end{enumerate}
There is no point repeating this argument here, 
and we refer to~\cite{cll,bcct} instead.
This settles the case of the $C=C_g$ equality. 
As for the $C=C_r$ equality,
we observe that
the above argument can
be carried out verbatim
once the mutliplicativity property
of the reversed Brascamp-Lieb constant is established.
This is the purpose of the rest of the article.
\begin{defi}
Given a proper subspace $E$ of $\R^n$ we 
let $B_{i,E}$ be the restriction of $B_{i}$ to $E$ and 
\[
B_{i,E^\perp} \colon x \in E^\perp \mapsto  q_i \circ  B_i x ,
\]
where $q_i$ is the orthogonal projection onto $(B_i E)^\perp$. 
Let $C_{r,E}$ be the reversed Brascamp-Lieb constant on $E$ associated to the datum 
\[
( c_1 , B_{1,E} ) , \dotsc , ( c_m , B_{m,E} )
\]
and $C_{r,E^\perp}$ be the Brascamp-Lieb constant on $E^\perp$ 
associated to the datum 
\[
( c_1 , B_{1 , E^\perp} ) , \dotsc , ( c_m , B_{m , E^\perp} )
\]
\end{defi}
\begin{rem}
It may happen that the restriction of $B_i$ to $E$ is identically $0$.
In the sequel, we take the convention that a Brascamp-Lieb 
datum is allowed to contain maps $B_i$ 
which are identically $0$, but that these are discarded
for the computation of the associated Brascamp-Lieb constants. 
\end{rem}
\begin{prop}
Let $E$ be a proper subspace of $\R^n$,
and assume that $E$ is critical, in the sense that 
\[
\dim ( E ) = \sum_{i=1}^m c_i \dim ( B_i E ) .
\]
Then $C_r = C_{r,E} \times C_{r,E^\perp}$.
\end{prop}
Bennett, Carbery, Christ and Tao proved
the corresponding property of $C$ and $C_g$,
we adapt their argument to prove the
multiplicativity of $C_r$. 
This adaptation is straightforward for the inequality
\[
C_r \leq C_{r,E} \times C_{r,E^\perp} 
\]
and is left to the reader (observe that 
criticality of $E$ is not even needed). 
We start the proof of the reversed inequality 
with a couple of simple observations.
\begin{lem}
\label{lem-app-un}
Upper semi-continuous functions having 
compact support saturate the reversed Brascamp-Lieb inequality. 
\end{lem}
\begin{proof} 
The regularity of the Lebesgue measure implies that 
given a non-negative integrable function $f_i$ 
on $\R^{n_i}$ and $\epsilon >0$
there exists a non-negative linear combination of
 indicators of compact sets $g_i$ 
 satisfying
\[
g_i \leq f_i \quad \text{and} \quad 
\int_{\R^{n_i}}  f_i  \, \drm x \leq (1+\epsilon) \int_{\R^{n_i}} g_i \, \drm x . 
\] 
The lemma follows easily. 
\end{proof}
The proof of the following lemma 
is left to the reader. 
\begin{lem}
\label{lem-app-deux}
If $f_1,\dotsc,f_m$ are upper semi-continuous functions 
on $\R^{n_1}, \dotsc , \R^{n_m}$ respectively, then 
the function $f$ defined on $\R^n$ by
\[
f (x) =  \sup  \Bigl( \prod_{i=1}^m f_i ( x_i )^{c_i} , \ \sum_{i=1}^m c_i B_i^* x_i = x Ê\Bigr) , 
\]
is upper semi-continuous as well. 
\end{lem}
\begin{rem}
If the Brascamp-Lieb datum happens to be \emph{degenerate},
in the sense that the map $(x_1 , \dotsc, x_m ) \mapsto \sum_{i=1}^m B_i^*x_i$
is not onto, then Brascamp-Lieb constants are easily
seen to be $+\infty$. Still the previous lemma
remains valid, provided the convention $\sup ( \emptyset )= 0$
is adopted. 
\end{rem}
Let us prove that 
$C_{r,E} \times C_{r,E^\perp} \leq C_r$. 
By Lemma~\ref{lem-app-un}, it is enough to prove that 
the inequality
\[
\prod_{i=1}^m \Bigl( \int_{B_iE} f_i \, \drm x \Bigr)^{c_i} 
\times \prod_{i=1}^m \Bigl( \int_{(B_i E) ^\perp} g_i \, \drm x \Bigr)^{c_i} 
\leq C_r  \Bigl( \int_E f \, \drm x \Bigr) \Bigl(  \int_{E^\perp} g \, \drm x \Bigr) . 
\]
holds for every compactly supported upper semi-continuous
functions $(f_i)_{i\leq m}$ and $(g_i)_{i\leq m}$, where $f$ and $g$
are defined by
\[
\begin{split}
f \colon & x\in E 
\mapsto  \sup \Bigl ( \prod_{i=1}^m f_i (  x_ i )^{c_i} , \ \sum_{i=1}^m c_i (B_{i,E} )^* x_i = x  \Bigr)  \\
g \colon & y \in E^\perp 
\mapsto  \sup \Bigl ( \prod_{i=1}^m g_i (  y_ i )^{c_i} , \ \sum_{i=1}^m c_i (B_{i,E^\perp} )^* x_i = y \Bigr) . 
\end{split}
\]
Let $\epsilon >0$. For $i\leq m$ define a function $h_i$ on $\R^{n_i}$ by
\[
 h_ i ( x + y ) = f_i ( x / \epsilon ) g_i (y ) , \quad \forall x\in B_i E , \, \forall y \in (B_iE)^\perp , 
\]
and let 
\[
h \colon z \in \R^n \mapsto \sup \Bigl ( \prod_{i=1}^m h_i (  z_ i )^{c_i} ,
    \ \sum_{i=1}^m c_i B_i^* z_i =z  \Bigr) .
\]
By definition of the reversed Brascamp-Lieb constant $C_r$
\begin{equation}
\label{step2-mult-rbl}
\prod_{i=1}^m \Bigl( \int_{\R^{n_i} } h_i \, \drm x \Bigr)^{c_i} \leq C_r \int_{\R^n} h \, \drm x . 
\end{equation}
Using the equality $\sum_{i=1}^m  c_i \dim ( B_i E ) = \dim ( E )$ we get
\[
\epsilon^{-\dim ( E ) } \prod_{i=1}^m  \Bigl( \int_{\R^{n_i} } h_i \, \drm x\Bigr)^{c_i }
= \prod_{i=1}^m  \Bigl( \int_{ B_i E } f_i \, \drm x \Bigr)^{c_i} \times
\prod_{i=1}^m \Bigl( \int_{ (B_i E)^\perp } g_i \, \drm x \Bigr)^{c_i} . 
\]
On the other hand, we let the reader check that 
for every $x\in E, y\in E^\perp$ 
\[
h ( \epsilon x  + y ) \leq f ( x) g_{\epsilon} (y)  ,
\]
where 
\[
g_\epsilon ( y ) = \sup \bigl( g (y') , \ \abs{y-y'} \leq K \epsilon \bigr) 
\]
and $K$ is a constant depending on the diameters of the supports of the functions $f_i$. 
Therefore 
\[
\epsilon^{-\dim E} \int_{\R^n} h \, \drm x
= \int_{E \times E^\perp} h( \epsilon x + y ) \, \drm x \drm y 
\leq \Bigl( \int_E f \, \drm x \Bigr) \Bigl( \int_{E^\perp} g_{\epsilon} \, \drm x \Bigr)  . 
\]
Inequality~\eqref{step2-mult-rbl} becomes
\[
\prod_{i=1}^m \Bigl( \int_{B_iE} f_i \, \drm x \Bigr)^{c_i} 
\times \prod_{i=1}^m \Bigl( \int_{(B_i E) ^\perp} g_i \, \drm x \Bigr)^{c_i} 
\leq C_r  \Bigl( \int_E f \, \drm x \Bigr) \Bigl(  \int_{E^\perp} g_{\epsilon} \, \drm x \Bigr)  . 
\]
Clearly $g$ has compact support, and $g$ is upper semi-continous 
by Lemma~\ref{lem-app-deux}. This implies easily that
\[
\lim_{\epsilon \rightarrow 0} \int_{E^\perp} g_{ \epsilon} \, \drm x = \int_{E^\perp} g \, \drm x ,
\]
which concludes the proof. 
\footnotesize

\end{document}